\documentclass[12pt]{amsart}

\usepackage{a4}
\usepackage{amsfonts}
\usepackage{amsmath}
\usepackage{amssymb}
\usepackage{mathrsfs}

\usepackage{xy}
\xyoption{all}

\usepackage[noBBpl]{mathpazo}

\usepackage{enumerate}

\renewcommand{\implies}{\Rightarrow}
\renewcommand{\iff}{\Leftrightarrow}

\newcommand{\ox}{\otimes}

\DeclareMathOperator{\Trd}{Trd} 
 
 \DeclareMathOperator{\tr}{tr}
 \DeclareMathOperator{\sig}{sig}

\newcommand{\<}{\langle}
\renewcommand{\>}{\rangle}
\newcommand{\x}{\times}

\DeclareMathOperator{\ind}{ind}
\DeclareMathOperator{\pind}{pind}

\newcommand{\Q}{\mathbb{Q}}
\newcommand{\R}{\mathbb{R}}

\newcommand{\N}{\mathbb{N}}

\newcommand{\vf}{\varphi}
\newcommand{\vt}{\vartheta}

\newcommand{\ga}{\alpha}

\newcommand{\s}{\sigma}

\newcommand{\ad}{\mathrm{ad}}
\newcommand{\id}{\mathrm{id}}

\newtheorem{lemma}{Lemma}[section]
\newtheorem{theorem}[lemma]{Theorem}
\newtheorem{prop}[lemma]{Proposition}

\theoremstyle{definition}

\theoremstyle{remark}
\newtheorem{remark}[lemma]{Remark}

\title{A hermitian analogue of the Br\"ocker--Prestel theorem}
\author{Vincent Astier}
\author{Thomas Unger}
\address{Fachbereich Mathematik und Statistik\\
Universit\"at Konstanz\\
D-78457 Konstanz\\ Germany}\email{vincent.astier@uni-konstanz.de}
\address{School of Mathematical Sciences\\ University College Dublin\\ Belfield\\
Dub\-lin~4\\ Ireland} \email{thomas.unger@ucd.ie}


\begin{document}

\begin{abstract} The Br\"ocker--Prestel local-global principle characterizes weak iso\-tropy of quadratic forms over a formally real field in terms of weak isotropy over the henselizations and isotropy over the real closures of that field. A hermitian analogue of this principle is presented for algebras of index at most two. An improved result is also presented for algebras with a decomposable involution, algebras of pythagorean index at most two,  and algebras over SAP and ED fields.
\end{abstract}

\maketitle

\section{Introduction}

In the algebraic theory of quadratic forms over fields the problem of determining whether a form is isotropic (i.e., has a non-trivial zero) has led to the development of several powerful local-global principles. They  allow one to test the isotropy of a form over the original field (``global'' situation) by testing it over a collection of other fields where the original problem is potentially easier to solve  (``local'' situation).

The most celebrated local-global principle is of course the Hasse--Minkowski theorem which 
gives a test for isotropy  over the rational numbers $\Q$ in terms of isotropy over the $p$-adic numbers $\Q_p$ for each prime $p$ and the real numbers $\R$. More generally, $\Q$ may be replaced by any global field $F$ and the collection $\bigl\{\{\Q_p\}_{p\text{ prime}},\ \R\bigr\}$ may be replaced by the collection of local fields $\{F_v\}$ (completions of $F$ with respect to all valuations $v$ on $F$). (See, for example, \cite[p.~170]{Lam}.)

Another landmark theorem in the theory of quadratic forms is Pfister's local-global principle. Here  a formally real field $F$ plays the role of ``global'' field, whereas the corresponding ``local'' fields are given by the collection of real closures $\{F_P\}$ where $P$ runs through the orderings of $F$. Pfister's local-global principle then says that a quadratic form $q$ is weakly hyperbolic over $F$ if and only if it is hyperbolic over $F_P$ for each ordering $P$ on $F$  or, in more familiar terminology, $q$ is a torsion form in the Witt ring $W(F)$ if and only if the signature of $q$ is zero for each ordering $P$ on $F$. (See, for example, \cite[p.~253]{Lam}.)

In this paper we are interested in a local-global principle for weak isotropy, due to Br\"ocker \cite{Bro} and Prestel \cite[Theorem~8.12]{P}. The role of  ``global'' field is again played  by a formally real field $F$, but this time the ``local'' fields are made up of the collection $\{F_v^H\}$ of henselizations of $F$ with respect to real valuations $v$ on $F$ (i.e., valuations with formally real residue field) together with  the collection $\{F_P\}$ of real closures of $F$ at orderings $P$ of $F$:

\begin{theorem}[Br\"ocker--Prestel]\label{BPThm}  Let $F$ be a formally real field  and let $q$ be a nonsingular quadratic form over $F$. Then $q$ is weakly isotropic over $F$ if and only if  the following conditions are satisfied:
\begin{enumerate}[\rm(i)]
\item $q\ox F_v^H$ is weakly isotropic for every real valuation $v$ on $F$;
\item $q \ox F_P$ is isotropic for every ordering $P$ on $F$.
\end{enumerate}
\end{theorem}

It is possible to state this theorem with weaker local conditions. For instance,  it suffices to consider only archimedean orderings, cf. \cite[Theorem~8.13]{P} or \cite[Theorem~6.3.1]{EP}.

Quadratic form theory has proved very useful in providing tools for studying central simple algebras with an involution and hermitian forms over division algebras, which can be thought of as ``twisted'' versions of quadratic form theory.
One can find ample examples of this phenomenon in the book \cite{BOI} and in the recent literature.

Given the usefulness of local-global principles in quadratic form theory it  seems natural to try to develop them in more generality.  The second author has been involved in several such projects \cite{LSU, LU, LUV}. In this paper we extend the Br\"ocker--Prestel local-global principle to algebras of index at most $2$ with an involution of the first kind over formally real fields.

The structure of this paper is as follows: in Section~\ref{prelim} we recall the basic facts about quadratic and hermitian forms and algebras with involution that will be needed. In Section~\ref{decomp} we show that, for central simple algebras of arbitrary index with decomposable involution of the first kind,  the problem at hand reduces to the generalized version of Pfister's local-global principle in \cite{LU}.
In Section~\ref{mainsec} we prove the hermitian version of the Br\"ocker--Prestel theorem for algebras of index at most $2$.  In Section~\ref{pindcase} we enlarge the class of algebras for which the theorem holds to those that become of index at most $2$ over the pythagorean closure of their centre. In Section~\ref{SAPcase} we show that the theorem holds for arbitrary central simple algebras with an involution of the first kind over a SAP field. Finally, in Section~\ref{EDcase}  we show that over an ED field the theorem reduces to the local-global principle studied in \cite{LSU}.

We thank Karim Becher for many discussions on this and related topics during Unger's visit to the University of Konstanz in March 2007. Unger also wishes to thank Alex Prestel for encouraging him to look at the Br\"ocker--Prestel theorem in a hermitian setting during his UCD visit in April~2004.

\section{Preliminaries}\label{prelim}

We assume that the reader is familiar with the basic notions from the theories of quadratic forms, hermitian forms and central simple algebras with an involution. We refer to the standard references \cite{Lam, Sch, BOI} for details.

Let $F$ be a field of characteristic different from $2$ and let $D$ be a central division algebra over $F$ equipped with an involution $\vt$ of the first kind. Let $\vf$ be a quadratic form over $F$ or, more generally, a hermitian or skew-hermitian form over $(D,\vt)$. The forms $\vf$ appearing in this paper are assumed to be nonsingular. In addition we exclude alternating forms over $F$ since they are hyperbolic and thus of no relevance to the problem at hand. Thus, our forms $\vf$ can always be diagonalized. 

For $n\in\N$ we denote the $n$-fold orthogonal sum 
\[\vf\perp\vf\perp\ldots\perp\vf\] 
simply by $n\x \vf$. We call $\vf$ \emph{weakly isotropic} or \emph{weakly hyperbolic} when there exists an $n\in\N$ such that $n\x \vf$ is isotropic or hyperbolic, respectively.
A weakly hyperbolic form is also called a \emph{torsion} form. A form which is not weakly isotropic is called \emph{strongly anisotropic}. 

Let $A$ be a central simple $F$-algebra, equipped with an involution $\s$ of the first kind. We call $(A,\s)$ \emph{isotropic} if there exists a nonzero $x\in A$ such that $\s(x)x=0$. The reader can verify that this is equivalent with the definition given in \cite[6.A]{BOI}. In particular, $(A,\s)$ is isotropic if there exists an idempotent $e\not=0$ in $A$ such that $\s(e)e=0$.

We call $(A,\s)$ \emph{weakly isotropic} if 
there exist nonzero $x_1,\ldots, x_n \in A$ such that 
\begin{equation}\label{wiso}
\s(x_1)x_1+\cdots+\s(x_n)x_n=0
\end{equation} 
and \emph{strongly anisotropic} otherwise (see \cite[Definition~2.2]{LSU}).

Assume now that the division algebra $D$ is Brauer equivalent to $A$. It follows from \cite[4.A]{BOI} that $\s$ is the \emph{adjoint involution} $\ad_h$ of some hermitian or skew-hermitian form $h$ over $(D,\vt)$. We then call $(A,\s)$  \emph{hyperbolic} if $h$ is a hyperbolic form. This is equivalent with the existence of a nonzero idempotent $e\in A$ such that $\s(e)=1-e$, cf. \cite[6.B]{BOI}. 

We call $(A,\s)$ \emph{weakly hyperbolic} if there exists an $n\in \N$ such that 
\[n\x (A,\s):= (M_n(F),t)\ox_F (A,\s)\] 
is hyperbolic, where $t$ denotes the transpose involution, see \cite[Definition~3.1]{LU}.

The reader can now easily verify that for $\s=\ad_h$ and for any $n\in\N$,
\begin{eqnarray} 
\nonumber n\x (A,\ad_h) &\cong &(M_n(A), \ad_{n\x h}),\\ 
n\x (A,\ad_h) \text{ is isotropic } &\iff & n\x h \text{ is isotropic} \iff \eqref{wiso} \text{ holds},\label{wiso2}\\
 \nonumber n\x (A,\ad_h) \text{ is hyperbolic } &\iff & n\x h \text{ is hyperbolic}
\end{eqnarray}
(see also \cite[Lemma~2.4]{LSU}).

Now let $F$ be a formally real field and let $q$ be a quadratic form over $F$. For an ordering $P$ on $F$ we denote the signature of $q$ at $P$ by $\sig_P q$. 

Let $(A,\s)$ be a central simple $F$-algebra with involution of the first kind. The \emph{signature of $\s$ at $P$} is defined by
\[\sig_P\s :=\sqrt{\sig_P T_\s},\]
where $T_\s$ is the \emph{involution trace form of} $(A,\s)$, which is a quadratic form on $A$ defined by 
\[T_\s(x):=\Trd_A(\s(x)x),\ \forall x\in A,\]
where $\Trd_A$ denotes the reduced trace of $A$,  see \cite[11.10]{BOI}.

We assume from now on that $F$ is a formally real field.

\section{The Decomposable Case}\label{decomp}

Let $A$ be a central simple $F$-algebra equipped with an involution $\s$ of the first kind. We call $(A,\s)$ \emph{decomposable} if we can write
\begin{equation}\label{eq1}
(A,\s)\cong (Q_1,\s_1)\ox_F \cdots \ox_F (Q_r,\s_r)\ox_F (M_s(F), t),
\end{equation}
for certain quaternion division algebras $Q_i$ with involution of the first kind $\s_i$, $1\leq i\leq r$ and $s\geq 1$. Here $t$ denotes again the transpose involution.

The following proposition shows that for such algebras the generalized version of Pfister's local-global principle \cite[Theorem~3.2]{LU} suffices to characterize weak isotropy.

\begin{prop} Let $(A,\s)$ be a decomposable algebra with involution over $F$. The following statements are equivalent:
\begin{enumerate}[\rm(i)]
\item\label{3.1i} $(A,\s)$ is weakly isotropic;
\item\label{3.1ii} $(A,\s)$ is weakly hyperbolic;
\item\label{3.1iii} $T_\s$ is weakly isotropic;
\item\label{3.1iv} $T_\s$ is weakly hyperbolic;
\item\label{3.1v}  $\sig_P\s=0$ for all orderings $P$ on $F$;
\item\label{3.1vi} $\sig_PT_\s=0$ for all orderings $P$ on $F$.
\end{enumerate}
\end{prop}

\begin{proof} Write $(A,\s)$ as in \eqref{eq1}. Since we are interested in when $(A,\s)$ is weakly isotropic or weakly hyperbolic, we may assume that $s=1$. The involution trace forms $T_{\s_i}$ are all 2-fold Pfister forms \cite[11.6]{BOI}. Since $T_\s=\bigotimes_{i=1}^r T_{\s_i}$ (this follows from the formula for the reduced trace of a tensor product), 
$T_\s$ is a Pfister form, which proves the equivalence of \eqref{3.1iii} and \eqref{3.1iv}. Obviously \eqref{3.1ii} $\implies$ \eqref{3.1i} $\implies$ \eqref{3.1iii}. By the definition of signature of an involution, 
 \eqref{3.1v} and  \eqref{3.1vi} are equivalent. 
Finally,  \eqref{3.1iv} and \eqref{3.1vi} are equivalent by Pfister's local-global principle and \eqref{3.1v} and \eqref{3.1ii} are equivalent by the generalized version of Pfister's local-global principle \cite[Theorem~3.2]{LU}.
\end{proof}

\section{Algebras of Index at Most $2$}\label{mainsec}

In this section we assume that the index of $A$ is at most $2$. This means that either $A$ is split, i.e., $A$ is isomorphic to a full matrix algebra over $F$, or $A$ is isomorphic to a full matrix algebra over a quaternion division $F$-algebra $D$.

\begin{remark}\label{remjac}  Let $a,b\in F^\times$ and let $D=(a,b)_F$ be a quaternion division algebra with quaternion conjugation $\gamma$.
The norm form of $D$ is $\<1,-a,-b,ab\>$. Let $h$ be a hermitian form over $(D,\gamma)$. Then $h\simeq \<\ga_1,\ldots, \ga_n\>$
and all $\ga_i \in F^\x$. To $h$ we can associate its trace form $q_h(x):=h(x,x)$ which is a quadratic form over $F$. It is clear that $h$ is (weakly) isotropic if and only if $q_h$ is (weakly) isotropic.  A simple computation shows that
\[q_h\simeq \<1,-a,-b,ab\>\ox  \<\ga_1,\ldots, \ga_n\>.\]
Furthermore, a well-known theorem of Jacobson \cite{Jac} says that, if $D$ is a division algebra, then the trace map 
\[\tr: S(D,\gamma)\to S(F),\ h\mapsto q_h\] 
is an injective morphism of Witt semi-groups of isometry classes of hermitian forms over $(D,\gamma)$ and quadratic forms over $F$, respectively.

When $D$ is split, $D\cong M_2(F)\cong (1,1)_F$, its norm form is hyperbolic and thus $q_h$ is always hyperbolic.
\end{remark}

\begin{theorem}\label{main}  Let $A$ be a central simple $F$-algebra of index $\leq 2$, equipped with an involution $\s$ of the first kind. Then $(A,\s)$ is weakly isotropic over $F$ if and only if  the following conditions are satisfied:
\begin{enumerate}[\rm(i)]
\item\label{loc1}  $(A\ox_F F_v^H, \s\ox\id_{F_v^H})$ is weakly isotropic for all real valuations $v$ on $F$;
\item\label{loc2} $(A\ox_F F_P, \s\ox\id_{F_P})$ is isotropic for all orderings $P$ on $F$.
\end{enumerate}
\end{theorem}

\begin{proof}  We treat the split case first: assume that $A$ is a full matrix algebra over $F$, $A=M_\ell(F)$.
If $\s$ is symplectic, then ($\ell$ is even and) $\s$ is adjoint to a skew-symmetric bilinear form over $F$ and is thus necessarily hyperbolic. Hence the statement is trivially true in this case.
If $\s$ is orthogonal, then $\s$ is adjoint to a quadratic form over $F$ and the result follows immediately from Theorem~\ref{BPThm}.

Next, assume that $\ind(A)=2$, i.e. that  $A=M_n(D)$, with $D$ a  quaternion division algebra over $F$. We will prove the nontrivial direction by contraposition. Thus, assume that $(A,\s)$ is strongly anisotropic. 

Let $\s$ be orthogonal. Then $\s=\ad_h$, where $h$ is a skew-hermitian form over $(D,\gamma)$. Let $K$ be a generic splitting field of $D$. It follows from \cite[Corollaire]{Dej} or \cite[Corollary~3.4]{PSS} that $(A,\s)$ remains strongly anisotropic after scalar extension to $K$.  But $A_K:=A\ox_F K$ is split. Thus, by the split case, there exists a real valuation $w$ on $K$ such that 
\[(A_K\ox_K K_w^H, \s_K\ox \id_{K_w^H})\cong (A\ox_F K_w^H, \s\ox\id_{K_w^H})\] 
is strongly anisotropic, or there exists an ordering $P$ on $K$ such that $(A_K\ox_K K_P, \s_K\ox \id_{K_P})$ is anisotropic.

If there is a real valuation $w$ on $K$ such that $(A\ox_F K_w^H, \s\ox\id_{K_w^H})$ is strongly anisotropic, let $v$ be the restriction of $w$ to $F$ and note that $F_v^H\subset K_w^H$. Consider the commutative diagram:
\begin{equation*}
\begin{aligned}[c] 
\xymatrix{
(A,\s) \ar[r]^--{\ox K} \ar[d]^--{\ox F_v^H}& (A\ox_F K, \s\ox \id_K)\ar[d]^--{\ox K_w^H}\\
(A\ox_F F_v^H,\s\ox\id_{F_v^H}) \ar[r]^--{\ox K_w^H} & (A\ox_F K_w^H, \s\ox\id_{K_w^H})\\
}
\end{aligned}
\end{equation*}
Now $(A\ox_F F_v^H,\s\ox\id_{F_v^H})$ has to be strongly anisotropic, for otherwise we would get a contradiction after scalar extension to $K_w^H$. 

This argument can be reproduced in the ordering case by considering real closures instead of henselizations.

Finally, suppose that $\s$ is symplectic. Then $\s=\ad_h$, where $h$ is a hermitian form over $(D,\gamma)$. By assumption $(A,\s)=(M_n(D), \ad_h)$ is strongly anisotropic. By \eqref{wiso2} this will be the case if and only if $h$ is strongly anisotropic over $(D,\gamma)$. By Remark~\ref{remjac} this will be the case if and only if the quadratic form $q_h$ is strongly anisotropic over $F$. Thus, by Theorem~\ref{BPThm}, 
there exists a real valuation $v$ on $F$ such that $q_h\ox F_v^H$ is strongly anisotropic over $F_v^H$, or there exists an ordering $P$ on $F$ such that $q\ox F_P$ is anisotropic over $F_P$.

If there exists a real valuation $v$ on $F$ such that $q_h\ox F_v^H$ is strongly anisotropic over $F_v^H$, consider the following commutative diagram:
\begin{equation*}
\begin{aligned}[c] 
\xymatrix{
S(D,\gamma) \ar[r]^--{\tr} \ar[d]^--{\ox F_v^H}& S(F)\ar[d]^--{\ox F_v^H}\\
S(D\ox_F F_v^H,\gamma\ox\id_{F_v^H}) \ar[r]^--{\tr} & S(F_v^H)\\
}
\end{aligned}
\end{equation*}
Since $q_h\ox F_v^H=q_{h\ox F_v^H}$, we obtain that $h\ox F_v^H$ is strongly anisotropic over $(D\ox_F F_v^H, \gamma\ox\id_{F_v^H})$. We observe that $D\ox_F F_v^H$ remains a division algebra, since otherwise $q_h\ox F_v^H$ would be  hyperbolic by Remark~\ref{remjac}, which would  contradict the strong anisotropy of $q_h\ox F_v^H$.
By \eqref{wiso2} we can conclude that $(M_n(D)\ox_F F_v^H, \ad_{h\ox F_v^H})$ is strongly anisotropic.

This argument can be repeated when there exists an ordering $P$ on $F$ such that $q\ox F_P$ is anisotropic upon replacing $F_v^H$ by $F_P$.
\end{proof}

Let $D$ be a division algebra, finite-dimensional over its centre $F$.
Morandi showed that a valuation $v$ on $F$ extends to $D$ if and only if $D\ox_F F_v^H$ is a division algebra \cite[Theorem~2]{Mor}. When this happens, we call the valuation $D$-\emph{admissible}. In a similar vein we call and ordering $P$ on $F$ $D$-\emph{admissible} if $D\ox_F F_P$ is a division algebra.

When $A=M_n(D)$ with $D$ a quaternion division algebra, the notion of $D$-admissibility can be used to restrict the set of local conditions needed to test for weak isotropy in  Theorem~\ref{main}:

In case $\s$ is symplectic, only $D$-admissible orderings and valuations are needed (as observed in the proof of the theorem). 

In case $\s$ is orthogonal we do not know whether a set of valuations smaller than the one provided by our use of the classical Br\"ocker--Prestel theorem 
would suffice. The set of orderings however is allowed to only consist of non-$D$-admissible ones. This can be seen as follows: since $K$ splits $D$, the norm form $N$ of $D$ becomes isotropic and thus hyperbolic over $K$. Consequently, $\sig_P N\ox K=0$ for all orderings $P$ on $K$ and hence $\sig_{P'} N=0$ where $P'$ denotes the restriction of $P$ to $F$ (since signatures do not change under scalar extension). Therefore, $N\ox {F_{P'}}$ is hyperbolic and $D\ox_F F_{P'}$ is split.

\begin{remark}  Some authors prefer to describe the local conditions involving valuations in Theorem~\ref{BPThm} in terms of residue forms over the residue fields, rather than extended forms over the henselizations. The reason for this is that the residue fields are in general much easier to work with than the henselizations.

When $A$ is not split, Theorem~\ref{main} can be reformulated as a local-global principle for weak isotropy of hermitian forms over $(D,\gamma)$ (when $\s$ is symplectic) or skew-hermitian forms over $(D,\gamma)$ (when $\s$ is orthogonal). In this way it resembles the statement for quadratic forms (Theorem~\ref{BPThm}) more closely.
In addition, each local condition involving a $D$-admissible valuation can also be stated in terms of residue (skew-)hermitian forms over the residue division algebras using \cite[Theorem~3.4]{Lar}.
\end{remark}

\begin{remark}  The key ingredient used in the proof of Theorem~\ref{main} is the index-two version of what is sometimes called the ``anisotropic splitting conjecture", which says: 
Let $(A, \s)$ be a central simple algebra with orthogonal  
involution over a field $F$, and let $K$ be a generic splitting field
of $A$. If $(A, \s)$ is anisotropic over $F$,  then $(A, \s)$ remains anisotropic over
$K$.

This conjecture is known to be true when the index of $A$ is one (this is trivial since $K/F$ is purely transcendental) or two (by \cite{PSS} or \cite{Dej}), and when $A$ is a division algebra (by \cite{Kar}). 
\end{remark}

\section{Algebras of Pythagorean Index at Most $2$}\label{pindcase}

Let $A$ be a central simple $F$-algebra and let $\widetilde{F}$ denote the pythagorean closure of $F$.
The \emph{pythagorean index} of $A$, $\pind(A)$, is defined by Becher \cite[\S4]{Bec} as follows:
\[\pind(A):=\ind(A\ox_F \widetilde{F}),\]
where $\ind$ denotes the (Schur-) index. For example, let $A=(a,b)_F\ox_F (c,d)_F$ be a biquaternion division algebra over $F$, where $a$ is a sum of squares in $F$, but not a square. Then $\ind(A)=4$, but $\pind(A)\leq 2$. 

The following lemma can be deduced immediately from \cite[Lemma~3.9]{LSU}:

\begin{lemma}\label{pythred} Let $(A,\s)$ be a central simple algebra with involution of the first kind over $F$. Then $(A\ox_F \widetilde{F}, \s\ox\id_{\widetilde{F}})$ is weakly isotropic if and only if $(A,\s)$ is weakly isotropic.
\end{lemma}

\begin{prop}\label{pind} Let $A$ be a central simple $F$-algebra of pythagorean index $\pind(A)\leq 2$, equipped with an involution $\s$ of the first kind. Then $(A,\s)$ is weakly isotropic over $F$ if and only if  the following conditions are satisfied:
\begin{enumerate}[\rm(i)]
\item\label{lloc1}  $(A\ox_F F_v^H, \s\ox\id_{F_v^H})$ is weakly isotropic for all real valuations $v$ on $F$;
\item\label{lloc2} $(A\ox_F F_P, \s\ox\id_{F_P})$ is isotropic for all orderings $P$ on $F$.
\end{enumerate}
\end{prop}

\begin{proof} The non-trivial direction can be obtained as follows: assume that the local conditions \eqref{lloc1} and \eqref{lloc2} are satisfied. Then they are satisfied for $(A\ox_F \widetilde{F}, \s\ox\id_{\widetilde{F}})$ since we can assume that $F_v^H\subseteq \widetilde{F}_w^H$ for every valuation $w$ on $\widetilde{F}$, where $v$ is the restriction of $w$ to $F$, and that $\widetilde{F}_P=F_{P_0}$ where $P_0=P\cap F$ for every ordering $P$ on $\widetilde{F}$.

Since by assumption $\ind(A\ox_F \widetilde{F}, \s\ox\id_{\widetilde{F}})\leq 2$ it follows from Theorem~\ref{main} that $(A\ox_F \widetilde{F}, \s\ox\id_{\widetilde{F}})$ is weakly isotropic. We may conclude by Lemma~\ref{pythred}.
\end{proof}

\section{Algebras over SAP Fields}\label{SAPcase} 

Let $F$ be a field, satisfying the \emph{Strong Approximation Property}, or SAP field for short. There are many ways of characterizing SAP fields (see \cite[Definition~3.1]{LSU} for an overview of all the equivalent definitions). For example, $F$ is SAP if and only if  for all $a,b\in F^\x$ the quadratic form $\<1,a,b,-ab\>$
is weakly isotropic.

\begin{prop} Let $A$ be a  central simple algebra equipped with an involution $\s$ of the first kind  over a SAP field $F$. 
Then $(A,\s)$ is weakly isotropic over $F$ if and only if  the following conditions are satisfied:
\begin{enumerate}[\rm(i)]
\item\label{llloc1}  $(A\ox_F F_v^H, \s\ox\id_{F_v^H})$ is weakly isotropic for all real valuations $v$ on $F$;
\item\label{llloc2} $(A\ox_F F_P, \s\ox\id_{F_P})$ is isotropic for all orderings $P$ on $F$.
\end{enumerate}
\end{prop}

\begin{proof} Let $\widetilde{F}$ be the pythagorean closure of $F$, as before. Then $A\ox_F \widetilde{F}$  is Brauer equivalent  to a quaternion division algebra $(-1,b)_{\widetilde{F}}$ for some $b\in \widetilde{F}^\x$, by \cite[Lemma~3.10]{LSU}. Hence, $\pind(A)\leq 2$ and we can conclude by Proposition~\ref{pind}.
\end{proof}

\begin{remark} In the proof above we use the fact that every algebra of exponent~$2$ over a SAP field $F$ has pythagorean index at most two. 
Karim Becher has kindly communicated to us that the converse also holds, resulting in yet another characterization of SAP fields: $F$ is SAP if and only if $\pind(A)\leq 2$ for any algebra $A$ of exponent $2$ over $F$.
\end{remark}

\section{Algebras over ED Fields.}\label{EDcase}  

Let $F$ be a field, satisfying the \emph{Effective Diagonalization Property}, or ED field for short. Again, there are several different ways of characterizing ED fields, but we will just use the original definition due to Ware \cite[\S2]{W}: a field $F$ is ED if and only
if every quadratic form over $F$ is effectively diagonalizable, i.e., every quadratic form $\vf$ over $F$ is isometric to a form $\<b_1,\ldots,b_n\>$ satisfying
$b_i\in P \implies b_{i+1}\in P$ for all $1\leq i <n$ and all orderings $P$ on $F$.

For algebras with involution of the first kind over ED fields, Theorem~\ref{main} reduces to the following simpler statement, which is a reformulation of the main theorem in \cite{LSU}.

\begin{prop}Let $F$ be an ED field and let $(A,\s)$ be a central simple algebra with involution of the first kind over $F$.
Then
$(A,\s)$ is weakly isotropic if and only if $(A\ox_F F_P, \s\ox\id_{F_P})$ is weakly isotropic for every ordering $P$ on $F$.
\end{prop}

\begin{proof} We show the non-trivial direction by contraposition.  
Assume that $(A, \s)$ is strongly anisotropic.
By \cite[Theorem~3.8]{LSU}  there exists an ordering $P$ of $F$ such that   $(A_P,\s_P):=(A\ox_F F_P,\s\ox \id_{F_P})$ is definite.

Since $\sig_P\s_P:=\sqrt{\sig_P T_{\s_P}}$, $T_{\s_P}$ is definite and thus strongly anisotropic over $F_P$. To show that $(A_P,\s_P)$ is strongly anisotropic, assume the opposite 
for the sake of contradiction.  Then we can find nonzero $x_1,\ldots, x_\ell \in A_P$ such that $\sum_{i=1}^\ell \s_P(x_i)x_i=0$. Taking the reduced trace of both sides gives $T_{\s_P}$ weakly isotropic, a contradiction.
\end{proof}

For quadratic forms a similar phenomenon happens in the presence of SAP fields: the Br\"ocker--Prestel theorem, Theorem~\ref{BPThm}, collapses to a theorem of Prestel stating that a quadratic form over a SAP field $F$ is weakly isotropic if and only if it is  isotropic over all real closures of $F$, cf. \cite[Theorem~9.1]{P}, \cite{Pres} and \cite{ELP}.

\end{document}